\documentclass[11pt]{amsart}

\usepackage{hyperref}
\usepackage[usenames]{color}
\usepackage{amsmath,amsthm,amsfonts,amssymb}

\newtheorem{theorem}{Theorem}[section]
\newtheorem{lemma}[theorem]{Lemma}
\newtheorem{corollary}[theorem]{Corollary}

\newtheorem{definition}[theorem]{Definition}
\newtheorem{conjecture}[theorem]{Conjecture}

\theoremstyle{definition}

\theoremstyle{remark}

\numberwithin{equation}{section}

\begin{document}

\title{On systems of equations in free abelian groups}

\author{Anton Menshov}
\address{Institute of Mathematics and Information Technologies\\Omsk State Dostoevskii University}
\curraddr{}
\email{menshov.a.v@gmail.com}
\thanks{}


\date{}

\begin{abstract}
In this paper we study the asymptotic probability that a random system of equations in free abelian group $\mathbb{Z}^m$ of rank $m$ is solvable.
Denote $SAT(\mathbb{Z}^m, k, n)$ and $SAT_{\mathbb{Q}^m}(\mathbb{Z}^m, k, n)$ the sets of all systems of $n$ equations in $k$ variables in the group $\mathbb{Z}^m$ solvable in $\mathbb{Z}^m$ and $\mathbb{Q}^m$ respectively.
We show that asymptotic density of the set $SAT_{\mathbb{Q}^m}(\mathbb{Z}^m, k, n)$ is equal to $1$ for $n \leq k$, and is equal to $0$ for $n > k$.
For $n < k$ we give nontrivial estimates for upper and lower asymptotic densities of the set $SAT(\mathbb{Z}^m, k, n)$.
When $n > k$ the set $SAT(\mathbb{Z}^m, k, n)$ is negligible.
Also for $n \leq k$ we provide some connection between asymptotic density of the set $SAT(\mathbb{Z}^m, k, n)$ and sums over full rank matrices involving their greatest divisors.
\end{abstract}

\maketitle

\section{Introduction}
\label{sec:intro}
In finite group theory the idea of genericity can be traced to works of Erd\"{o}s and Turan \cite{ET} and Dixon \cite{D}.
Nowadays it is an area of active research.
In geometric group theory the generic approach is due to Gromov \cite{G1, G2, G3} and is associated with random walks on groups.

Recently, a host of papers appeared on generic properties of individual groups. We mention here, in particular, results on generic properties of one-relator groups \cite{KS1, KS2} and averaged Dehn functions \cite{KR, R}.

Gilman, Myasnikov and Roman'kov studied satisfiability of random equations in free abelian and finitely generated nilpotent groups \cite{GMR_nilpotent} and in free groups \cite{GMR_free}.
Probability of a homogeneous equation in a surface group to have solutions is studied in \cite{ACV}.

In this paper we study satisfiability of random systems of equations in free abelian groups of finite rank and extend corresponding results obtained in \cite{GMR_nilpotent} for equations.

Denote $SAT(\mathbb{Z}^m, k, n)$ and $SAT_{\mathbb{Q}^m}(\mathbb{Z}^m, k, n)$ the sets of all systems of $n$ equations in $k$ variables in the group $\mathbb{Z}^m$ solvable in $\mathbb{Z}^m$ and $\mathbb{Q}^m$ respectively.
We study asymptotic density of the sets above with respect to the natural stratification of the group $\mathbb{Z}^m$ with balls corresponding to the uniform norm $\| \cdot  \|_{\infty}$ of Euclidean space $\mathbb{R}^m$.
In the sequel we treat $\mathbb{Z}^m$ and its subgroups as integer lattices in $\mathbb{R}^m$.

In Section \ref{sec:rational-systems}, using the asymptotics for the number of integral matrices of fixed rank \cite{K}, we show in theorem \ref{TH_SAT_IN_Qm_AD} that the set $SAT_{\mathbb{Q}^m}(\mathbb{Z}^m, k, n)$ is generic if $n \leq k$, and negligible if $n > k$.

In Section \ref{sec:ehrhart} we recall some basic notions concerning lattice points counting in rational polytopes and extend  inequality \cite[Theorem 6]{BM} derived for coefficients of Ehrhart polynomials to Ehrhart quasipolynomials (see theorem \ref{TH_EHR_COEFF_BOUND}).
Using this inequality, in theorem \ref{TH_GCD_SUMS} we establish connection between asymptotic density of the set $SAT(\mathbb{Z}^m, k, n)$ and sums of inverse greatest divisors over full rank matrices.
Based on this result, we make a conjecture \ref{CON_1} concerning asymptotic density of the set $SAT(\mathbb{Z}^m, n, n)$.
In theorem \ref{TH_MAIN_BOUNDS} we give nontrivial estimates for upper and lower asymptotic densities of the set $SAT(\mathbb{Z}^m, k, n)$.
Namely, we show that for $n < k$ this densities lie in the interval from $\left( \prod_{j = k - n + 1}^{k} \zeta(j) \right)^{-1}$ to $\left( \frac{\zeta(k+m)}{\zeta(k)} \right)^n$, where $\zeta(s)$ is Riemann zeta-function.

\section{Preliminaries}
\label{sec:prelim}

\subsection{Asymptotic density}
\label{subsec:density}
A {\em stratification} of a countable set $T$ is a  sequence $\{T_r\}_{r \in \mathbb{N}}$ of non-empty finite subsets $T_r$ whose union is $T$. 
Stratifications are often specified by length functions. A {\em length function} on $T$ is a map $l: T \to \mathbb{N}$ from $T$ to the nonnegative integers $\mathbb{N}$ such that the inverse image of every integer is finite. The corresponding spherical and ball stratifications are formed by {\em spheres} $S_r = \{x \in T \mid l(x) = r\}$ and {\em balls} $B_r = \{x \in T \mid l(x) \leq r\}$.

\begin{definition}
The asymptotic density of $M \subset T$ with respect to a stratification $\{T_r\}$ is defined to be
\[
\rho(M) = \lim_{r\to\infty} \rho_r(M), \ \ \ \mbox{where} \ \ \  \rho_r(M) = \frac{|M \cap T_r|} {|T_r|}
\]
 when the limit exists. Otherwise, we use the limits 
\[
\bar{\rho}(M) = \limsup_{r\to\infty} \rho_r(M), \ \ \ \underline{\rho}(M) = \liminf_{r\to\infty} \rho_r(M),
\]
and call them upper and lower asymptotic densities respectively.
\end{definition}

$M$ is said to be {\em generic} in $T$ with respect to the stratification $\{T_r\}_{r \in \mathbb{N}}$ if $\rho(M) = 1$ and {\em negligible} if $\rho(M) = 0$. A property of elements of $T$ is generic if it holds on a generic subset.

Asymptotic density is one of the tools for measuring sets in infinite groups (see \cite{BMS} for details).

Let $\mathbb{Z}^m$ be a free abelian group of rank $m$.
We identify $\mathbb{Z}^m$ with the standard integer lattice in Euclidean space $\mathbb{R}^m$.
We assume that $\mathbb{R}^m$ is equipped with the uniform norm defined for $v = (v_1, \dots, v_m)$ by the formula
\[
\|v\|_{\infty} = \max_{i} \{ |v_i| \}.
\]
This norm induces the lenght function $\|\cdot\|_{\infty}: \mathbb{Z}^m \to \mathbb{N}$ with balls
\[
B_r^m = \{ v \in \mathbb{Z}^m \mid \|v\|_{\infty} \leq r \}.
\]
Further we will compute asymptotic density of some subsets in free abelian groups with respect to the given ball stratification.

\subsection{Equations in groups}
\label{subsec:equations}
An equation $u=1$ in $k$ variables over a group $G$ is an expression of the form
\[
    g_{0}x_{i_{1}}^{m_{1}}g_{1} \dots x_{i_{n}}^{m_{n}}g_{n} = 1,
\]
where each coefficient $g_j$ is a group element, each exponent $m_j$ is an integer, and each $x_{i_j}$ is taken from an alphabet of variables, $X = \{x_1, \dots , x_k\}$.
In this case, the free product $G_{X} = F(X) * G$ is the space of all equations in variables $X$ and coefficients in $G$.
A solution of $u=1$ in $G$ is an assignment $x_j\to h_j \in G$ such that $g_0h_{i_1}^{m_1}g_1\ldots h_{i_n}^{m_n}g_n = 1$.
Denote by $SAT(G,K)$ the set of all equations from $G_X$ which have a solution in $G$ ({\em satisfiable} equations).

Since we study equations in free abelian groups, we will use additive notation throughout.
In line with this we will write equation as
\begin{equation}\label{F_EQ_IN_Zm}
    \gamma_{1}\mathbf{x_1} + \dots + \gamma_{k}\mathbf{x_k} = \mathbf{b},
\end{equation}
where $\gamma_j \in \mathbb{Z}$ and $\mathbf{b} \in \mathbb{Z}^m$ are fixed, and $\mathbf{x_j} = (x_{j1}, \dots, x_{jm})$ are variables.

System
\begin{equation}
\left\{
    \begin{alignedat}{4}
        \gamma_{11}\mathbf{x_1} &+ \cdots &&+ \gamma_{1k}\mathbf{x_k} &&= \mathbf{b_1}, \\
        \;\vdots  &            &&\quad\vdots \\
        \gamma_{n1}\mathbf{x_1} &+ \cdots &&+ \gamma_{nk}\mathbf{x_k} &&= \mathbf{b_n},
    \end{alignedat}
\right.
\end{equation}
of $n$ equations of the form (\ref{F_EQ_IN_Zm}) will be written as
\begin{equation}\label{F_SYS_IN_Zm}
    AX = B,
\end{equation}
where
\[
A = (\gamma_{ij}) \in \mathbb{Z}^{nk}, \quad
B = \left(
        \begin{array}{c}
                \mathbf{b_1} \\
                \vdots \\
                \mathbf{b_n}
        \end{array}
    \right) \in \mathbb{Z}^{nm}, \quad \text{and} \quad
X = \left(
        \begin{array}{c}
                \mathbf{x_1} \\
                \vdots \\
                \mathbf{x_k}
        \end{array}
    \right).
\]

It is natural to consider two equations as essentially the same if one can be transformed into the other by applying identities of the variety of abelian groups.
So the natural space of equations in variables $X = \{x_1, \dots, x_k\}$ over a free abelian group $\mathbb{Z}^m$ is the direct product $ A(X) \times \mathbb{Z}^m \simeq \mathbb{Z}^{k+m}$ of a free abelian group $A(X)$ with basis $X$ and group $\mathbb{Z}^m$.
By $G_{X,n}$ we denote the space of all systems of $n$ equations from $G_X$ and $\mathbb{Z}^{n(k+m)}$ is a natural choice for it.
We denote $SAT(G,k,n)$ the set of all systems from $G_{X,n}$ solvable in $G$.
Systems from $\mathbb{Z}^m_{X,n}$ solvable in $\mathbb{Q}^m$ will be denoted by $SAT_{\mathbb{Q}^m}(\mathbb{Z}^m, k, n)$.

Observe, that in (\ref{F_SYS_IN_Zm}) different coordinates of variables are independent from each other, so the following obvious  lemma holds.

\begin{lemma}\label{L_SYS_IN_Zm_SOLUTION}
A system $AX = B$ of the form~(\ref{F_SYS_IN_Zm}) is solvable in $\mathbb{Z}^{m} (\mathbb{Q}^{m})$ if and only if systems $Ax = B_{i}$ are solvable over $\mathbb{Z} (\mathbb{Q})$ for any $i = 1, \dots, m$, where $B_i$ is the $i$-th column of the matrix $B$.

If $X_i = (x_{1i}, \dots, x_{ki})^T$ is a solution of $Ax = B_{i}$, then $X = (X_1, \dots, X_m) $ is a solution of $AX = B$.
\end{lemma}

\section{Systems solvable in $\mathbb{Q}^m$}
\label{sec:rational-systems}
In this section we will compute asymptotic density of the set $SAT_{\mathbb{Q}^m}(\mathbb{Z}^m,k,n)$ of all systems of $n$ equations in $k$ variables in $\mathbb{Z}^m$ solvable in $\mathbb{Q}^m$.

Consider a system of linear diophantine equations
$Ax = b$ with $A \in \mathbb{Z}^{nk}, b \in \mathbb{Z}^{n}$.
According to Kronecker-Capelli theorem the system above has solution over $\mathbb{Q}$ if and only if the rank of its coefficient matrix $A$ is equal to the rank of augmented matrix $(A|b)$.

We state here the main result of \cite{K} where the asymptotics for the number of integral matrices of fixed rank is derived.
Denote
\begin{align*}
    V_{n, k, s}(\mathbb{Z}) &= \{ A \in \mathbb{Z}^{nk} \mid rank(A) =s \}, \\
    N(r; n, k, s) &= \left| \{ A \in V_{n, k, s}(\mathbb{Z}) \mid \|A\|_2 < r \} \right|,
\end{align*}
where $\|A\|_{2} = \sqrt{(\sum_{i,j}a_{ij}^{2})}$.

\begin{theorem}[Katznelson \cite{K}]\label{TH_RANK}
For $k \geq n > s \geq 1$ and as $r$ tends to infinity:
\begin{enumerate}
    \item[(1)] for $n<k$, \ \
        $N(r; n, k, s) = \alpha(n, k, s)r^{ks} + O(r^{ks - 1})$.
    \item[(2)] for $n=k$, \ \
        $N(r; k, k, s) = \beta(k, s)r^{ks}\log r + O(r^{ks})$.
\end{enumerate}
\end{theorem}

Observe, that the growth rate of $N(r; n, k, s)$ doesn't change if we replace the norm $\|\cdot\|_2$ by $\|\cdot\|_{\infty}$.
Denote
\[
    N(r; n, k, s)_{\infty} = \left| \{ A \in V_{n, k, s}(\mathbb{Z}) \mid \|A\|_{\infty} < r \} \right|.
\]
It is easy to see that
\[
    \dfrac{1}{\sqrt{nk}} \|A\|_{2} \leq \|A\|_{\infty} \leq \|A\|_{2}
\]
and
\[
    N(r; n, k, s) \leq N(r; n, k, s)_{\infty} \leq N(r\sqrt{nk}; n, k, s).
\]
Therefore for $s = 1, \dots, n-1$
\begin{equation}\label{F_RANK_DENSITY}
    \rho(V_{n, k, s}(\mathbb{Z})) =
    \lim_{r \to \infty} \dfrac{N(r; n, k, s)_{\infty}}{(2r+1)^{nk}} = 0
\end{equation}
which implies that $\rho(V_{n, k, n}(\mathbb{Z})) = 1$, i.e., asymptotically almost all $n \times k$ matrices have full rank.

\begin{theorem}\label{TH_SAT_IN_Qm_AD}
The set $SAT_{\mathbb{Q}^m}(\mathbb{Z}^m, k, n)$ is generic if $n \leq k$ and negligible if $n > k$.
\end{theorem}
\begin{proof}
Denote $B_r = B_{r}^{n(k+m)}$.
Consider the set
\[
    S_1 = \{ (A|B) \in \mathbb{Z}^{n(k+m)} \mid rank(A)=n \}
\]
of all systems $AX=B$ of the form (\ref{F_SYS_IN_Zm}) for $n \leq k$.
All systems from $S_1$ are solvable in $\mathbb{Q}^m$, hence the following inclusion holds
$
    S_1 \subset SAT_{\mathbb{Q}^m}(\mathbb{Z}^m, k, n).
$
Consider the projection $\pi_{1}: \mathbb{Z}^{n(k+m)} \to \mathbb{Z}^{nk}$, defined by $\pi_{1}(A|B) = A$.
Observe, that
\[
    \pi_{1}(S_{1} \cap B_r) =
    V_{n, k, n}(\mathbb{Z}) \cap \pi_{1}(B_r)
\]
and each preimage contains $(2r+1)^{nm}$ elements. Hence
\[
    \dfrac{ | \pi_{1}( S_{1} \cap B_r) | }{(2r+1)^{nk}} =
        \rho_{r}(V_{n, k, n}(\mathbb{Z}))
\]
and
\[
    \rho_{r}(S_{1}) =
    \dfrac{ | S_{1} \cap B_r| }{ |B_r| } =
    (2r+1)^{nm}\dfrac{ | \pi_{1}(S_{1} \cap B_r) | }{(2r+1)^{n(k+m)}} =
    \rho_{r}(V_{n, k, n}(\mathbb{Z})).
\]
Therefore $\rho(S_{1}) = \rho(V_{n, k, n}(\mathbb{Z})) = 1$,
which implies that for $n \leq k$
\[
    \rho(SAT_{\mathbb{Q}^m}(\mathbb{Z}^m, k, n)) = 1.
\]

Next consider the set
\[
    S_2 = \{ (A|B) \in \mathbb{Z}^{n(k+m)} \mid rank(A|B_1)<k+1 \}
\]
of all systems $AX=B$ of the form (\ref{F_SYS_IN_Zm}) for $n > k$, where $B_1$ is the first column of the matrix $B$.
According to lemma \ref{L_SYS_IN_Zm_SOLUTION} solvability of the system $AX = B$ in $\mathbb{Q}^m$ implies solvability of the system $Ax = B_1$ over $\mathbb{Q}$, so $rank(A|B_1) < k+1$ since $n > k$.
Thus the following inclusion holds
$
    SAT_{\mathbb{Q}^m}(\mathbb{Z}^m, k, n) \subset S_2.
$
Consider the projection $\pi_{2}: \mathbb{Z}^{n(k+m)} \to \mathbb{Z}^{n(k+1)}$, defined by $\pi_{2}(A|B) = (A|B_{1})$.
Observe, that
\[
    \pi_{2}(S_{2} \cap B_r) =
    \pi_{2}(B_r) \backslash 
    (V_{k+1, n, k+1}(\mathbb{Z}) \cap \pi_{2}(B_r))
\]
and each preimage contains $(2r+1)^{n(m-1)}$ elements.
Hence
\[
    \dfrac{ | \pi_{2}( S_{2} \cap B_r) | }{(2r+1)^{n(k+1)}} =
        1 - \rho_{r}(V_{k+1, n, k+1}(\mathbb{Z}))
\]
and
\[
    \rho_{r}(S_{2}) =
    \dfrac{|S_{2} \cap B_r|}{|B_r|} =
    (2r+1)^{n(m-1)}\dfrac{|\pi_{2}(S_{2} \cap B_r)|}{(2r+1)^{n(k+m)}} =
    1 - \rho_{r}(V_{k+1, n, k+1}(\mathbb{Z})).
\]
Therefore $\rho(S_{2}) = 1 - \rho(V_{k+1, n, k+1}(\mathbb{Z})) = 0$, which implies that for $n>k$
\[
    \rho(SAT_{\mathbb{Q}^m}(\mathbb{Z}^m, k, n)) = 0.
\]
\end{proof}

\begin{corollary}
The set $SAT(\mathbb{Z}^m, k, n)$ is negligible if $n>k$.
\end{corollary}
\begin{proof}
Since $SAT(\mathbb{Z}^m, k, n) \subset SAT_{\mathbb{Q}^m}(\mathbb{Z}^m, k, n)$, if follows that for $n>k$
$
    \rho(SAT(\mathbb{Z}^m, k, n)) = 0.
$
\end{proof}

\section{Lattice points in rational polytopes}
\label{sec:ehrhart}
Counting lattice points in the integral dilates of a subset of Euclidean space $\mathbb{R}^d$ is a well known problem.
For rational polytopes this problem has been studied in the $1960$s by the French mathematician Eug\`{e}ne Ehrhart.
These results may be useful when computing asymptotic densities in free abelian groups (see, for example, \cite{M}).
Detailed survey of further results can be found in \cite{BR} (see also \cite{STANLEY}).
We recall some basic notions first.
A {\em convex polytope} in $\mathbb{R}^d$ is a finite intersection of closed half-spaces, i.e.,
\[
\mathcal{P} = \{ \mathbf{x} \in \mathbb{R}^d \mid A\mathbf{x} \leq b \}, \ \ where \ \ A \in \mathbb{R}^{md}, \; b \in \mathbb{R}^m.
\]
This definition is called the {\em hyperplane description} of $\mathcal{P}$.
Also any bounded convex polytope in $\mathbb{R}^d$ has the {\em vertex description} and could be presented as the convex hull of finitely many points in $\mathbb{R}^d$.
This vertex description of a polytope is equivalent to the hyperplane description.

The {\em dimension} of a polytope $\mathcal{P}$ is the dimension of the affine space
\[
\mathrm{span}~\mathcal{P} = \{ \mathbf{x} + \lambda (\mathbf{y}-\mathbf{x}) \mid \mathbf{x}, \mathbf{y} \in \mathcal{P}, \lambda \in \mathbb{R} \}
\]
spanned by $\mathcal{P}$.
If $\mathcal{P}$ has dimension $d$, we use the notation $\dim \mathcal{P} = d$ and call $\mathcal{P}$ a $d$-polytope.

A bounded convex polytope $\mathcal{P}$ is called {\em integral} if all of its vertices have integral coordinates, and $\mathcal{P}$ is called {\em rational} if all of its vertices have rational coordinates.
We will call the least common multiple of the denominators of the coordinates of the vertices of $\mathcal{P}$ the {\em denominator} of $\mathcal{P}$.

For $t \in \mathbb{Z}^{+}$ denote $t \mathcal{P} = \{ t\mathbf{x} \mid \mathbf{x} \in \mathcal{P} \}$ the $t^{th}$ dilate of $\mathcal{P}$.
We denote the {\em lattice-point enumerator} for the $t^{th}$ dilates of $\mathcal{P}$ by
\[
L_{\mathcal{P}}(t) = |t \mathcal{P} \cap \mathbb{Z}^d|.
\]
We define the {\em Ehrhart series} of $\mathcal{P}$ as the generating function of $L_{\mathcal{P}}(t)$
\[
    Ehr_{\mathcal{P}}(z) = \sum_{t \geq 0} L_{\mathcal{P}}(t)z^t.
\]
Here we assume that $L_{\mathcal{P}}(0) = 1$.

By $\binom{n}{k}$ we denote the \emph{binomial coefficient}, defined through
\begin{equation}\label{F_BINOM_COEFF}
    \binom{n}{k} = \dfrac{n(n-1) \dots (n-k+1)}{k!}
\end{equation}
for $n \in \mathbb{C}, \; k \in \mathbb{Z}^+$.

The following theorem states some known results concerning lattice points in integral polytopes \cite[theorems 3.8, 3.12 and lemmas 3.13, 3.14]{BR}.

\begin{theorem}\label{TH_EHR_LATTICE}
If $\mathcal{P}$ is an integral convex $d$-polytope, then
\begin{enumerate}
    \item[(1)]
        $Ehr_{\mathcal{P}}(z) = \dfrac{h_dz^d + \dots + h_1z + 1}{(1-z)^{d+1}}$ with $h_i \in \mathbb{N}$,
    \item[(2)]
        $L_{\mathcal{P}}(t) = \binom{t+d}{d} + h_1 \binom{t+d-1}{d} + \dots + h_{d-1} \binom{t+1}{d} + h_d \binom{t}{d}$ is a polynomial in $t$ of degree $d$.
\end{enumerate}
$L_{\mathcal{P}}(t)$ is called the Ehrhart poynomial of $\mathcal{P}$.
\end{theorem}

Some coefficients of $L_{\mathcal{P}}(t) = c_d t^d + \dots + c_1 t + c_0$ have geometric interpretation.
For example, the leading coefficient $c_d$ is equal to $d$-dimensional volume of $\mathcal{P}$  \cite[corollary 3.20]{BR}, which implies that
\begin{equation}\label{F_VOLUME_POLYNOM}
    \mathrm{vol}(\mathcal{P}) = \dfrac{1}{d!} (h_d + \dots + h_1 + 1).
\end{equation}
It is also known that $c_0 = 1$ \cite[corollary 3.15]{BR}.

We recall that a {\em quasipolynomial} $Q$ is an expression of the form $Q(t) = c_n(t)t^n + \dots + c_1(t)t + c_0(t)$, where $c_0, \dots, c_n$ are periodic functions in $t$ and $c_n$ is not the zero function.
The {\em degree} of $Q$ is $n$, and the least common period of $c_0, \dots, c_n$ is the {\em period} of $Q$.
Alternatively, for a quasipolynomial $Q$, there exist a positive integer $k$ and polynomials $p_0, p_1, \dots, p_{k-1}$ such that $Q(t) = p_i(t)$ if $t \equiv i \pmod k$.
The minimal such $k$ is the period of $Q$.

Results of the theorem \ref{TH_EHR_LATTICE} can be extended to rational polytopes \cite[theorem 3.23, ex. 3.25]{BR}.

\begin{theorem}\label{TH_EHR_RATIONAL}
If $\mathcal{P}$ is a rational convex $d$-polytope with the denominator $p$, then
\begin{enumerate}
    \item[(1)]
        $Ehr_{\mathcal{P}}(z) = \dfrac{\sum_{i=0}^{d(p+1)-1} h_iz^i}{(1-z^p)^{d+1}}$ with $h_i \in \mathbb{N}$,
    \item[(2)]
        $L_{\mathcal{P}}(t)$ is a quasipolynomial in $t$ of degree $d$ and its period divides $p$.
\end{enumerate}
$L_{\mathcal{P}}(t)$ is called the Ehrhart quasipolynomial of $\mathcal{P}$.
\end{theorem}

We note that there is an algorithm by Alexander Barvinok to compute Ehrhart quasipolynomials. Barvinok's algorithm is polynomial in fixed dimension, it has been implemented in the software package $\mathsf{LattE}$ \cite{L}.

We will need an explicit formula for Ehrhart quasipolynomials, similar to one specified in theorem \ref{TH_EHR_LATTICE} for Ehrhart polynomials.

\begin{lemma}\label{L_EHR_QUASIPOLYNOM_FORMULA}
In notations of the theorem \ref{TH_EHR_RATIONAL}
\[
    L_{\mathcal{P}}(t) = f_j(t) \ \ \text{if} \ t \equiv j \pmod p
\]
where
\[
    f_j(t) = h_j \binom{t+d}{d} + h_{p+j} \binom{t+d-1}{d} + \dots + h_{dp+j} \binom{t}{d}.
\]
\end{lemma}
\begin{proof}
\begin{align*}
    Ehr_{\mathcal{P}}(z) &=
        \dfrac{\sum\limits_{i=0}^{d} \sum\limits_{j=0}^{p-1} h_{ip+j} z^{ip+j} }{(1-z^p)^{d+1}} \\ &=
        \left( \sum_{i=0}^{d} \sum_{j=0}^{p-1} h_{ip+j} z^{ip+j} \right) \sum_{t \geq 0} \binom{t+d}{d} z^{tp} \\ &=
        \sum_{i=0}^{d} \sum_{j=0}^{p-1} h_{ip+j} \sum_{t \geq 0} \binom{t+d}{d} z^{ip + j + tp} \\ &=
        \sum_{i=0}^{d} \sum_{j=0}^{p-1} h_{ip+j} \sum_{t \geq i} \binom{t+d-i}{d} z^{pt + j}.
\end{align*}
In all infinite sums we can start the index $t$ with $0$.
Hence
\begin{align*}
    Ehr_{\mathcal{P}}(z) &=
        \sum_{i=0}^{d} \sum_{j=0}^{p-1} h_{ip+j} \sum_{t \geq 0} \binom{t+d-i}{d} z^{pt + j} \\ &=
        \sum_{j=0}^{p-1} \sum_{t \geq 0} \left( \sum_{i=0}^{d} h_{ip+j}  \binom{t+d-i}{d} \right) z^{pt + j} \\ &=
        \sum_{j=0}^{p-1} \sum_{t \geq 0} f_j(t) z^{pt + j}.
\end{align*}
\end{proof}

For Ehrhart quasipolynomials $L_{\mathcal{P}}(t) = c_d(t)t^d + \dots + c_1(t)t + c_0(t)$ the leading coefficient $c_d(t) $ is equal to $d$-dimensional volume of $\mathcal{P}$ and $c_0(t) = 1$.
So $c_d$ and $c_0$ are constants.
This implies the formular similar to (\ref{F_VOLUME_POLYNOM})
\begin{equation}\label{F_VOLUME_QUASIPOLYNOM}
    \mathrm{vol}(\mathcal{P}) = \dfrac{1}{d!} (h_j + h_{p+j} + \dots + h_{dp+j}),
\end{equation}
for $j = 0, \dots, p-1$.

The coefficients of Ehrhart polynomials are very special, see \cite{EHR_COEFF_ROOTS} for example. 
We will be interested in the following inequality.

\begin{theorem}[Betke, McMullen \cite{BM}]
If $\mathcal{P}$ is an integral convex $d$-polytope with Ehrhart polynomial $L_{\mathcal{P}}(t) = c_d t^d + \dots + c_1 t + 1$, then
\[
    c_r \leq (-1)^{d-r} s(d, r) c_d + (-1)^{d-r-1} \dfrac{s(d, r+1)}{(d-1)!} \ \ \ \text{for} \ r = 1, \dots, d-1,
\]
where $s(i,j)$ denote the Stirling numbers of the first kind.
\end{theorem}

We recall that \emph{Stirling numbers of the first kind} $s(n, k)$ are defined through
\begin{equation}\label{F_STIRLING_NUMBERS}
    [x]_n = x(x-1) \dots (x-n+1) = \sum_{k=0}^{n} s(n, k) x^k.
\end{equation}
Now we extend the inequality above to rational polytopes.

\begin{theorem}\label{TH_EHR_COEFF_BOUND}
If $\mathcal{P}$ is a rational convex $d$-polytope with the denominator $p$ and Ehrhart quasipolynomial
$L_{\mathcal{P}}(t) = c_d t^d + c_{d-1}(t) t^{d-1} + \dots + c_1(t) t + 1$, then
\[
    |c_r(t)| \leq |s(d+1, r+1)| c_d \ \ \ \text{for} \
    r = 1, \dots, d-1, \ t \in \mathbb{Z}^+
\]
\end{theorem}
\begin{proof}
For arbitrary polynomial $g(x)$ by $\left. g(x) \right|_i$ we denote the coefficient of $x^i$ in $g(x)$.

According to lemma \ref{L_EHR_QUASIPOLYNOM_FORMULA} we have $L_{\mathcal{P}}(t) = f_j(t)$ if $t \equiv j \pmod p$, where
\[
    f_j(t) = \sum_{i=0}^{d} h_{ip+j} \binom{t+d-i}{d}.
\]
Then
\[
    c_r(j) = \left. f_j(t) \right|_r =
    \left. \left( \sum_{i=0}^{d} h_{ip+j} \binom{t+d-i}{d} \right) \right|_r =
    \sum_{i=0}^{d} h_{ip+j} \left. \binom{t+d-i}{d} \right|_r.
\]
It is easy to see that for $i = 1, \dots, d$
\[
    \left| \left. \binom{t+d-i}{d} \right|_r \right| \leq \left. \binom{t+d}{d} \right|_r.
\]
From  (\ref{F_VOLUME_QUASIPOLYNOM}), since $h_i \in \mathbb{N}$, it follows that for $j = 0, \dots, p-1$
\begin{align*}
    |c_r(j)| &= | \left. f_j(t) \right|_r | \\ &\leq
    \sum_{i=0}^{d} h_{ip+j} \left| \left. \binom{t+d-i}{d} \right|_r \right| \\ &\leq
    \sum_{i=0}^{d} h_{ip+j} \left. \binom{t+d}{d} \right|_r \\ &=
    \left. \binom{t+d}{d} \right|_r \left( h_j + h_{p+j} + \dots + h_{dp+j} \right) \\ &=
    \left. \binom{t+d}{d} \right|_r d! \; vol(\mathcal{P}) =
    |s(d+1, r+1)| c_d.
\end{align*}
\end{proof}

Let $\mathbf{w_1}, \dots, \mathbf{w_d}$ be lineary independent vectors in $\mathbb{R}^d$.
The set
\[
\Lambda = \Lambda(\mathbf{w_1}, \dots, \mathbf{w_d}) =
\{ \alpha_1 \mathbf{w_1} + \dots + \alpha_d \mathbf{w_d} \mid \alpha_i \in \mathbb{Z} \}
\]
is called the {\em lattice} with basis $\{\mathbf{w_1}, \dots, \mathbf{w_d}\}$.
The number
\[
d(\Lambda) = | \det (\mathbf{w_1}, \dots, \mathbf{w_d}) |
\]
is called the {\em determinant} of the lattice.

A convex polytope $\mathcal{P}$ is called {\em lattice polytope} (with respect to the lattice $\Lambda$) if all of its vertices belongs to $\Lambda$.

We note that results of this section remains true if we replace the standard integer lattice $\mathbb{Z}^d$ by an arbitrary lattice $\Lambda(\mathbf{w_1}, \dots, \mathbf{w_d})$.
Indeed, consider the matrix $A = (\mathbf{w_1}, \dots, \mathbf{w_d})$ formed by $\mathbf{w_i}$ as columns and let $\psi$ be the linear transformation corresponding to $A$.
Then $\Lambda = \psi(\mathbb{Z}^d)$ and $\mathbb{Z}^d = \psi^{-1}(\Lambda)$.
If $\mathcal{P}$ is a rational $d$-polytope with respect to the basis $\{\mathbf{w_1}, \dots, \mathbf{w_d}\}$, then $\psi^{-1}(\mathcal{P})$ is a rational $d$-polytope with respect to the standard basis of $\mathbb{R}^d$ and
\[
    L_{\mathcal{P}, \Lambda}(t) =
    \left| \{ t\mathcal{P} \cap \Lambda \} \right| =
    \left| \{ t\psi^{-1}(\mathcal{P}) \cap \mathbb{Z}^d \} \right|.
\]
Given that $\mathrm{vol}(\mathcal{P}) = |\det(A)| \mathrm{vol}(\psi^{-1}(\mathcal{P}))$ we get that the leading coefficient of $L_{\mathcal{P}, \Lambda}(t)$ is equal to $\frac{\mathrm{vol}(\mathcal{P})}{d(\Lambda)}$.

\section{Systems solvable in $\mathbb{Z}^m$}
\label{sec:integral-systems}


Let $M \in \mathbb{Z}^{nk} (n \leq k)$ and $rank(M) = n$, then by $\gcd(M)$ we denote the \emph{greatest divisor} of $M$, defined as the greatest common divisor of the determinants of $M$.
The determinants of a matrix are, of course, the determinants of the greatest square matrices contained in it.
A matrix $M$ is called {\em unimodular} if $\gcd(M)=1$.

First we recall some classical criterions telling whether a linear diophantine systems of the form
\begin{equation}\label{F_SYS_IN_Z_2}
    Ax = b \ \ \text{with} \
    A \in \mathbb{Z}^{nk}, \;
    b \in \mathbb{Z}^{n},
\end{equation}
has integral solution.

\begin{theorem}[Smith \cite{SMITH}]\label{TH_SMITH}
Let $Ax = b$ be a system of the form (\ref{F_SYS_IN_Z_2}) with $rank(A) = n$.
Then the system has integral solution if and only if the greatest divisors of its augmented and unaugmented matrices are equal.
\end{theorem}

We also mention a criterion due to Van der Waerden.

\begin{theorem}[Van der Waerden]\label{TH_VAN_DER_WAERDEN}
A system $Ax = b$ of the form (\ref{F_SYS_IN_Z_2}) has integral solution if and only if for every $v \in \mathbb{Q}^n$ such that $vA \in \mathbb{Z}^k$, $(v, b) \in \mathbb{Z}$.
\end{theorem}

However, we will use other criterion wich allows us to apply results of Section \ref{sec:ehrhart}.

For $n \leq k$ consider a system from $\mathbb{Z}^m_{X, n}$ of the form
\begin{equation}\label{F_SYS_IN_Zm_2}
    AX = B \ \ \text{with} \ A \in \mathbb{Z}^{nk}, B \in \mathbb{Z}^{nm}.
\end{equation}
Denote $A_i$ and $B_j$ the columns of $A$ and $B$ respectively, and let $H_A = \left\langle A_1, \dots, A_k \right\rangle \leq \mathbb{Z}^n$ be the subgroup generated by the columns of $A$.
Then the following obvious lemma holds.

\begin{lemma}\label{L_SOLVABILITY_CRITERION_Zm}
A system $AX=B$ of the form (\ref{F_SYS_IN_Zm_2}) is solvable in $\mathbb{Z}^m$ if and only if $B_i \in H_A$ for every $i = 1, \dots, m$.
\end{lemma}

If $rank(A) = n$ then $H_A$ is of finite index, thus $H_A$ is a $n$-dimensional lattice in $\mathbb{R}^n$ with determinant $d(H_A) = |\mathbb{Z}^n : H_A| = \gcd(A)$.
Therefore
\[
    \left| \{ B_r^n \cap H_A \} \right| =
    \left| \{ r B_1^n \cap H_A \} \right| =
    L_{B_1^n, H_A}(r)
\]
is Ehrhart quasipolynomial for $B_1^n$, which further will be denoted by $L_A(r)$.
Consider the sum
\begin{equation}\label{F_MAIN_SUM}
    S_{m, k, n}(r) = \sum_{\substack{ A \in B_r^{nk} \\ rank(A) = n }} L_{A}^{m}(r).
\end{equation}
According to lemma \ref{L_SOLVABILITY_CRITERION_Zm} this sum describes the number of solvable systems of the form (\ref{F_SYS_IN_Zm_2}) with $rank(A)=n$ in the ball $B_r^{n(k+m)}$.
Let
\begin{gather}
    L_A(t) = c_{A, n} t^n + c_{A, n-1}(t) t^{n-1} + \dots + c_{A, 1}(t) t + 1, \nonumber \\
    L_{A}^{m}(r) = \sum_{i = 0}^{mn} \alpha_{A, i}r^i, \nonumber
\end{gather}
then $S_{m, k, n}(r)$ can be presented as
\[
    S_{m, k, n}(r) =
    \sum_{\substack{ A \in B_r^{nk} \\ rank(A) = n }} \sum_{i = 0}^{mn} \alpha_{A, i}r^i =
    \sum_{i = 0}^{mn} \left( \sum_{\substack{ A \in B_r^{nk} \\ rank(A) = n }}  \alpha_{A, i} \right) r^i =
    \sum_{i = 0}^{mn} s_{m,k,n,i}(r) r^i.
\]
It is easy to see that $\alpha_{A, 0} = 1$ and
\[
    \alpha_{A, mn} = c_{A, n}^{m} = \left( \dfrac{\mathrm{vol}(B_1^n)}{d(H_A)} \right)^m = 2^{mn} \gcd(A)^{-m},
\]
so
\[
    s_{m, k, n, mn}(r) = \sum_{\substack{ A \in B_r^{nk} \\ rank(A) = n }} \alpha_{A, mn} =
    2^{mn} \sum_{\substack{ A \in B_r^{nk} \\ rank(A) = n }} \gcd(A)^{-m}.
\]

\begin{lemma}\label{L_MAIN_SUM_LEADING_COEFF_O}
As $r$ tends to infinity $s_{m, k, n, mn}(r) = O(r^{nk})$.
\end{lemma}

\begin{proof} Observe, that
\begin{align*}
    s_{m, k, n, mn}(r) &=
    2^{mn} \sum_{\substack{ A \in B_r^{nk} \\ rank(A) = n }} \gcd(A)^{-m} \\ &\leq
    2^{mn} \sum_{\substack{ A \in B_r^{nk} \\ rank(A) = n }} 1 \\ &\leq
    2^{mn} \left| B_r^{nk} \right| = 2^{mn} (2r + 1)^{nk}.
\end{align*}
\end{proof}

The following lemma shows that $r^{mn} s_{m, k, n, mn}(r)$ makes a major contribution to the sum $S_{m, k, n}(r)$ from the asymptotic point of view.

\begin{lemma}\label{L_MAIN_SUM_OTHER_COEFF}
$\lim\limits_{r \to \infty} \dfrac{s_{m,k,n,i}(r) r^i}{r^{n(k+m)}} = 0$ for $i = 0, \dots, mn-1$.
\end{lemma}
\begin{proof}
The case $i = 0$ is obvious since $\alpha_{A, 0} = 1$.

According to the theorem \ref{TH_EHR_COEFF_BOUND} for $r = 1, \dots, n-1$
\[
    |c_{A, r}(t)| \leq \max_{i} \{ |s(n+1, i)| \} c_{A, n},
\]
then for any sum of the form
\begin{equation}\label{F_MIXED_SUM}
    \Omega(r) = \sum_{\substack{ A \in B_r^{nk} \\ rank(A) = n }} c_{A, i_1}(r) \dots c_{A, i_s}(r),
    \ \ \text{where} \ i_j \in \{ 1, \dots, n \}
\end{equation}
we have
$
    |\Omega(r)| \leq \beta s_{s, k, n, sn}(r),
$
for some constant $ \beta $, which doesn't depend on $r$.
So $\Omega(r)=O(r^{nk})$.
Any sum
\[
    s_{m, k, n, i}(r) =
    \sum_{\substack{ A \in B_r^{nk} \\ rank(A) = n }}  \alpha_{A, i}
\]
is a finite sum of the sums of the form (\ref{F_MIXED_SUM}), so it is also $O(r^{nk})$.

Statement of the lemma now follows from the fact that for any function $f(r) = O(r^{nk})$
\[
    \lim_{r \to \infty} \dfrac{f(r)r^i}{r^{n(k+m)}} = 0 \ \ \text{if} \ i < mn.
\]
\end{proof}

Denote
\[
    F_{m, k, n}(r) = \dfrac{s_{m, k, n, mn}(r)}{2^{mn}}=
    \sum_{\substack{ A \in B_r^{nk} \\ rank(A) = n }} \gcd(A)^{-m}.
\]
We now establish a connection between $F_{m, k, n}(r)$ and  asymptotic density of the set $SAT(\mathbb{Z}^m, k, n)$.

\begin{theorem}\label{TH_GCD_SUMS}
$\rho(SAT(\mathbb{Z}^m, k, n)) = \rho$ if and only if
$
    \lim\limits_{r \to \infty} \dfrac{F_{m,k,n}(r)}{(2r)^{nk}} = \rho.
$
\end{theorem}
\begin{proof}
If asymptotic density of the set $SAT(\mathbb{Z}^m, k, n)$ exists, then according to (\ref{F_RANK_DENSITY}) only systems with full rank matrices contribute to asymptotic density, so
\[
    \lim_{r \to \infty} \dfrac{S_{m,k,n}(r)}{|B_r^{n(k+m)}|} = \rho.
\]
According to lemma \ref{L_MAIN_SUM_OTHER_COEFF} only the leading term of $S_{m,k,n}(r)$ contribute to the limit above, hence
\[
    \rho =
    \lim_{r \to \infty} \dfrac{s_{m,k,n,mn}(r)r^{mn}}{|B_r^{n(k+m)}|} =
    \lim_{r \to \infty} \dfrac{s_{m,k,n,mn}(r)r^{mn}}{(2r)^{n(k+m)}} =
    \lim\limits_{r \to \infty} \dfrac{F_{m,k,n}(r)}{(2r)^{nk}}.
\]
Arguing as above one can show that the converse is true.
\end{proof}

Now we consider the case when $n=k$ and $m=1$. Then
\begin{equation}\label{F_INVERSE_DET_SUM}
    F_{1, n, n}(r) = \sum_{\substack{ A \in B_r^{nn} \\ rank(A) = n }} \dfrac{1}{|\det(A)|}.
\end{equation}
In \cite{DRS} asymptotics for the number of integral matrices of fixed determinant is derived.
Let
\begin{align*}
    V_{n, k} &= \{ M \in \mathbb{Z}^{nn} \mid \det(M) = k \}, \\
    N(r, V_{n,k}) &= \left| \{ M \in V_{n,k} \mid \|M\|_2 \leq r \} \right|,
\end{align*}
then according to \cite[example 1.6]{DRS}
\[
    N(r, V_{n,k}) \sim c_{n,k}r^{n^2 - n}.
\]
If we assume that each value occurs (roughly) equally often, then
\[
    F_{1,n,n}(r) = O(r^{n^2 - n}\ln(r)),
\]
so
\[
    \lim\limits_{r \to \infty} \dfrac{F_{1,n,n}(r)}{(2r)^{nn}} = 0
\]
and $\rho(SAT(\mathbb{Z}, n, n)) = 0$.
Since
\[
    \rho(SAT(\mathbb{Z}^{m+1}, n, n)) \leq \rho(SAT(\mathbb{Z}^m, n, n)),
\]
then $\rho(SAT(\mathbb{Z}^m, n, n)) = 0$ for arbitrary $m$.

This allows us to formulate the following conjecture.
\begin{conjecture}\label{CON_1}
$F_{1,n,n}(r) = O(r^{n^2-n}\ln(r))$ and $\rho(SAT(\mathbb{Z}^m,n,n))=0$.
\end{conjecture}

In the rest of this section we give some nontrivial estimates for upper and lower asymptotic densities of the set $SAT(\mathbb{Z}^m, k, n)$.

Denote $U_{n,k}$ the set of all $n \times k$ unimodular integer matrices.
Below, $\zeta(s) = \sum_{n=1}^{\infty} \frac{1}{n^s}$ is Riemann zeta-function.

Asymptotic density of rectangular unimodular integer matrices is derived in \cite{MRW}.
\begin{theorem}[Maze, Rosenthal, Wagner \cite{MRW}]\label{TH_UNIMOD_AD}
\mbox{}
\begin{enumerate}
    \item[(1)]
        $\rho(U_{n,k}) = \left( \prod\limits_{j = k - n + 1}^{k} \zeta(j) \right)^{-1}$ if $k > n \geq 1$,
    \item[(2)]
        $\rho(U_{n,n}) = 0$ if $n \geq 1$.
\end{enumerate}
\end{theorem}

We note that \cite{MRW} uses a slightly different notion of natural density in $\mathbb{Z}^{nk}$, which is nevertheless equivalent to our definition of asymptotic density with respect to balls stratification.

Theorem \ref{TH_UNIMOD_AD} allows us to bound lower asymptotic density of the set $SAT(\mathbb{Z}^m, k, n)$, since every system of the form (\ref{F_SYS_IN_Zm_2}) with unimodular matrix $A$ is solvable in $\mathbb{Z}^m$.
To bound upper asymptotic density we need to know asymptotic density of solvable equations in $\mathbb{Z}^m$.

\begin{theorem}[Gilman, Myasnikov, Roman'kov \cite{GMR_nilpotent}]
\label{TH_SAZ_IN_Z_AD}
\mbox{}
\begin{enumerate}
    \item[(1)]
        $\rho(SAT(\mathbb{Z}^m, k)) = \dfrac{\zeta(k+m)}{\zeta(k)}$
            if $k \geq 2$,
    \item[(2)]
        $\rho(SAT(\mathbb{Z}^m, 1)) = 0$.
\end{enumerate}
\end{theorem}

Now we are ready to prove the following theorem.

\begin{theorem}\label{TH_MAIN_BOUNDS}
The following inequalities holds
\mbox{}
\begin{enumerate}
    \item[(1)]
        $\rho(U_{n,k}) \leq
            \underline{\rho}(SAT(\mathbb{Z}^m, k, n))$
        if $k > n > 1$,
    \medskip
    \item[(2)]
        $\overline{\rho}(SAT(\mathbb{Z}^m, k, n)) \leq
            \rho(SAT(\mathbb{Z}^m, k))^{n}$
        if $k \geq n > 1$.
\end{enumerate}
\end{theorem}
\begin{proof}
Consider the set
\[
    S_{1} = \{ (A|B) \in \mathbb{Z}^{n(k+m)} \mid A \ \text{is unimodular} \}
    \subset SAT(\mathbb{Z}^m, k, n)
\]
when $k > n > 1$.
It is easy to show that $\rho(S_{1}) = \rho(U_{n,k})$. Then
\[
    \rho(U_{n,k}) \leq \underline{\rho}(SAT(\mathbb{Z}^m, k, n)).
\]

Next consider the set $S_{2}$ of all systems from $\mathbb{Z}^m_{X,n}$, where each equation is solvable.
Clearly $SAT(\mathbb{Z}^m, k, n) \subset S_{2}$. Observe, that
\[
    \rho_{r}(S_{2}) = \rho_{r}(SAT(\mathbb{Z}^m, k))^{n}
\]
and
\[
    \rho(S_{2}) =
    \lim_{r \to \infty}\rho_{r}(S_{2}) =
    \left( \lim_{r \to \infty}\rho_{r}(SAT(\mathbb{Z}^m, k)) \right)^{n} =
    \rho(SAT(\mathbb{Z}^m, k))^{n}.
\]
Then
\[
    \overline{\rho}(SAT(\mathbb{Z}^m, k, n)) \leq \rho(SAT(\mathbb{Z}^m, k))^{n}.
\]
\end{proof}

\section*{Acknowledgements}
The author would like to thank Professor V.~A.~Roman'kov for his valuable comments and suggestions.


\begin{thebibliography}{99}
\bibitem{ET} Erdos P., Turan P. \textit{On some problems of statistical group theory I} // Z. Wahrscheinlichkeitstheorie verw. Geb., 1965, V. 4, P. 175--186

\bibitem{D} Dixon J. \textit{The probability of generating the symmetric group} // Math. Z., 1969, V. 110, N. 3, P. 199--205.

\bibitem{G1} Gromov M. \textit{Hyperbolic Groups} // In: Essays in Group Theory,  MSRI publ. 8, (1987), P. 75–-263.

\bibitem{G2} Gromov M. \textit{Asymptotic invariants of infinite groups} // Geometric group theory, vol. 2 (Sussex, 1991), 1--295, London Math. Soc. Lecture Note Ser., vol. 182, Cambridge Univ. Press, Cambridge, 1993.

\bibitem{G3} Gromov M. \textit{Random walks in random groups} // Geom. Funct. Analysis, 2003, V. 13, P. 73--146.

\bibitem{KS1} Kapovich I., Schupp P. \textit{Genericity, the Arzhantseva-Olshanskii method and the isomorphism problem for one-relator groups} // Math. Ann., 2005, V. 331, N. 1, P. 1–-19.

\bibitem{KS2} Kapovich I., Schupp P. \textit{Delzant’s T-ivariant, one-relator groups and Kolmogorov complexity} // Comment. Math. Helv., 2005, V. 80, P. 911--933.

\bibitem{KR} Kukina E. G., Roman'kov V. A. \textit{Subquadratic growth of the averaged Dehn function for free Abelian groups} // Siberian Mathematical J., 2003, V. 44, N. 4, P. 605--610. 

\bibitem{R} Roman'kov V. A. \textit{Asymptotic growth of averaged Dehn functions for nilpotent groups} // Algebra and Logic, 2007, V. 46, N. 1, P. 37--45.

\bibitem{GMR_nilpotent} Gilman R., Myasnikov A., Roman'kov V. \textit{Random equations in nilpotent groups} // J. of Algebra, 2012, V. 352 P. 192--214.

\bibitem{GMR_free} Gilman R., Myasnikov A., Roman'kov V. \textit{Random equations in free groups} // Groups – Complexity – Cryptology, 2011, V. 3, P. 257--284.

\bibitem{ACV} Antolin Y., Ciobanu L., Viles N. \textit{On the asymptotics of visible elements and homogeneous equations in surface groups} // Groups, Geometry and Dynamics, 2012, V. 6, P. 619--638.

\bibitem{BMS} Borovik A. V., Myasnikov A. G., Shpilrain V. \textit{Measuring sets in infinite groups} // Computational and Statistical Group Theory, Contemporary Math, Amer. Math. Soc. Providence, RI, 298, P. 21--42.

\bibitem{K} Katznelson Y. \textit{Integral Matrices of Fixed Rank} // Proc. Amer. Math. Soc., 1994, V. 120, N. 3, P. 667--675.

\bibitem{BR} Beck M., Robins S. \textit{Computing the Continuous Discretely} // Springer, 2007.

\bibitem{STANLEY} Stanley R. \textit{Enumerative combinatorics} // Vol. 1, Cambridge university press, 1997. 

\bibitem{L} Software package LattE: Lattice-Point Enumeration, Available at \href{https://www.math.ucdavis.edu/~latte/}{https://www.math.ucdavis.edu/latte/}

\bibitem{BM} Betke U. McMullen P. \textit{Lattice points in lattice polytopes} // Monatshefte für Mathematik, 1985, V. 99, N. 4, P. 253--265.

\bibitem{M} Menshov A. V. \textit{Asymptotic density of rational sets in free abelian groups} // \href{http://arxiv.org/abs/1401.6558}{arXiv:math.GR/1401.6558}

\bibitem{EHR_COEFF_ROOTS} Beck M., De Loera J., Develin M., Pfeifle J., Stanley R. \textit{Coefficients and roots of Ehrhart polynomials} In: Integer Points in Polyhedra-Geometry, Number Theory, Algebra, Optimization, volume 374 of Contemp. Math., pages 15–36. Amer. Math. Soc., Providence, RI, 2005. \href{http://arxiv.org/abs/math/0402148}{arXiv:math.CO/0402148}.

\bibitem{SMITH} Smith H. J. S. \textit{On Systems of Linear Indeterminate Equations and Congruences} // Philosophical Transactions of the Royal Society of London, 1861, V. 151, P. 293--326.

\bibitem{DRS} Duke W., Rudnick Z., Sarnak P. \textit{Density of integer points on affine homogeneous varieties} // Duke Math Journal, 1993, V. 71, N. 1, P. 143--179.

\bibitem{MRW} Maze G., Rosenthal J., Wagner U. \textit{Natural density of rectangular unimodular integer matrices} // Linear Algebra and Its Applications, 2011, V. 434, N. 5, P. 1319--1324.

\end{thebibliography}
\end{document}